\documentclass[oneside]{amsart}
\usepackage{amsmath,amsfonts,amssymb}%
\usepackage[all]{xy}
\usepackage[usenames]{color}
\usepackage{setspace}
\usepackage{comment}
\usepackage{url}

\usepackage{amsmath, amsthm, amsfonts, amssymb, stmaryrd}%
\usepackage{mathrsfs}
\usepackage{enumerate}

\usepackage{xparse, expl3}

\usepackage[capitalize]{cleveref}  
\crefname{lemma}{Lemma}{Lemmas}
\crefname{corollary}{Corollary}{Corollaries}
\crefname{theorem}{Theorem}{Theorems}
\crefname{equation}{Equation}{Equations}
\crefname{example}{Example}{Examples}
\crefname{section}{Section}{Sections}
\crefname{subsection}{Section}{Sections}
\newcommand\rest[1][]{\EM{\upharpoonleft_{#1}}}

\newcommand{\defn}[1]{{\bf{#1}}}

\def\Powerset{\mathcal{P}}

\newcommand{\Erdos}{\text{Erd\H{o}s}}

\DeclareMathOperator{\Rel}{Rel}
\DeclareMathOperator{\arity}{ar}

\DeclareDocumentCommand{\SeqSize}{d[] d()}
{
\EM{#1^{(#2)}}
}

\DeclareMathOperator{\SubStrop}{Sub}
\DeclareDocumentCommand{\SubStr}{d[] d()}
{
\IfNoValueTF{#1}
	{
        \EM{\SubStrop(#2)}
	}
	{
        \EM{\SubStrop_{#1}(#2)}
	}
}

\DeclareMathOperator{\SFop}{SF}
\DeclareDocumentCommand{\SF}{d[]}
{
\IfNoValueTF{#1}
	{
	\EM{\SFop}
	}
	{
	\EM{\SFop_{#1}}
	}
}

\DeclareDocumentCommand{\dual}{d()}
{
\IfNoValueTF{#1}
	{
	\EM{\hat{\ }}
	}
	{
	\EM{\hat{#1}}
	}
}

\DeclareDocumentCommand{\SizeAS}{d()}
{
\IfNoValueTF{#1}
	{
	\EM{|\cdot|}
	}
	{
	\EM{|#1|}
	}
}

\DeclareDocumentCommand{\compcK}{d[]}
{
\IfNoValueTF{#1}
	{
	\EM{\mathbb{K}}
	}
	{
	\EM{\mathbb{K}[#1]}
	}
}

\DeclareDocumentCommand{\AgeK}{d[]}
{
\IfNoValueTF{#1}
	{
	\EM{\mathbf{K}}
	}
	{
	\EM{\mathbf{K}[#1]}
	}
}

\DeclareDocumentCommand{\Rel}{d[]}
{
\IfNoValueTF{#1}
	{
	\EM{\mathcal{R}}
	}
	{
	\EM{\mathcal{R}_{#1}}
	}
}

\DeclareDocumentCommand{\Func}{d[]}
{
\IfNoValueTF{#1}
	{
	\EM{\mathcal{F}}
	}
	{
	\EM{\mathcal{F}_{#1}}
	}
}

\DeclareDocumentCommand{\ar}{d[]}
{
\IfNoValueTF{#1}
	{
	\EM{\arity}
	}
	{
	\EM{\arity_{#1}}
	}
}

\DeclareDocumentCommand{\Fn}{d<> d[] d()}
{
\IfNoValueTF{#3}
	{
	\EM{\textrm{Fn}(#1, #2)}
	}
	{
	\EM{\textrm{Fn}(#1, #2, #3)}
	}
}

\DeclareMathOperator{\Sym}{Sym}
\DeclareDocumentCommand{\Perm}{d()}
{
\EM{\Sym(#1)}
}

\DeclareDocumentCommand{\tdcl}{d[] d()}
{
\IfNoValueTF{#1}
{
    \EM{\textbf{term}(#2)}
}
{
    \EM{\textbf{term}_{#1}(#2)}
}
}

\DeclareDocumentCommand{\gdcl}{d[] d()}
{
\IfNoValueTF{#1}
{
    \EM{\textrm{gcl}(#2)}
}
{
    \EM{\textrm{gcl}_{#1}(#2)}
}
}

\DeclareDocumentCommand{\qftp}{d[] d()}
{
\IfNoValueTF{#1}
{
    \EM{\quantfreetp(#2)}
}
{
    \EM{\quantfreetp_{#1}(#2)}
}
}

\DeclareMathOperator{\quantfreetp}{qtp}

\DeclareDocumentCommand{\Closure}{d<> d()}
{
\IfNoValueTF{#1}
{
    \EM{\textrm{cl}(#2)}
}
{
    \EM{\textrm{cl}_{#1}(#2)}
}
}

\DeclareDocumentCommand{\ClosureMap}{d[] d()}
{
\EM{\textrm{clMap}(#1, #2)}
}

\DeclareDocumentCommand{\CHP}{d()}
{
\IfNoValueTF{#1}
{
    \textrm{(CHP)}
}
{
    \textrm{(\EM{#1}-CHP)}
}
}

\DeclareDocumentCommand{\CJEP}{d()}
{
\IfNoValueTF{#1}
{
    \textrm{(CJEP)}
}
{
    \textrm{(\EM{#1}-CJEP)}
}
}

\makeatletter
\newcommand{\dotminus}{\mathbin{\text{\@dotminus}}}

\newcommand{\@dotminus}{%
  \ooalign{\hidewidth\raise1ex\hbox{.}\hidewidth\cr$\m@th-$\cr}%
}

\DeclareMathOperator{\Lang}{\mathscr{L}}

\def\cA{{\EM{\mathcal{A}}}}

\def\cF{{\EM{\mathcal{F}}}}

\def\cM{{\EM{\mathcal{M}}}}
\def\cN{{\EM{\mathcal{N}}}}

\def\w{\EM{\omega}}

\def\^{\EM{{}^{\And}}}

\def\And{\EM{\wedge}}

\def\<{\EM{\langle}}
\def\>{\EM{\rangle}}

\def\EM#1{\ensuremath{#1}}

\def\ol#1{\EM{\overline{#1}}}

\def\st{\,:\,}
\def\:{\colon}

\providecommand{\dotdiv}{% Don't redefine it if available
  \mathbin{% We want a binary operation
    \vphantom{+}% The same height as a plus or minus
    \text{% Change size in sub/superscripts
      \mathsurround=0pt % To be on the safe side
      \ooalign{% Superimpose the two symbols
        \noalign{\kern-.35ex}% but the dot is raised a bit
        \hidewidth$\smash{\cdot}$\hidewidth\cr % Dot
        \noalign{\kern.35ex}% Backup for vertical alignment
        $-$\cr % Minus
      }%
    }%
  }%
}

\DeclareDocumentCommand{\RightJustify}{m}{\hspace*{\fill}\mbox{#1}\penalty-9999\relax}

\newcounter{margincounter}
\DeclareDocumentCommand{\displaycounter}{}
	{{\arabic{margincounter}}}
\DeclareDocumentCommand{\incdisplaycounter}{}
	{{\stepcounter{margincounter}\arabic{margincounter}}}

\DeclareDocumentCommand{\DeclareComment}{m m m o d()}{%
\expandafter\DeclareDocumentCommand\csname Hide#1\endcsname {}
	{%
	\expandafter\DeclareDocumentCommand\csname #1\endcsname {+m} {}
	
	\expandafter\DeclareDocumentCommand\csname f#1\endcsname {+m} {}

	\expandafter\DeclareDocumentEnvironment{e#1} {} {} {}
	}

\expandafter\DeclareDocumentCommand\csname Show#1\endcsname {}
	{
	\expandafter\DeclareDocumentCommand\csname #1\endcsname {+m}
		{%
		\textcolor{#2}
			{ 
			{\tiny \bf (#3)}
			\IfValueT{#5}
				{%
				#5
				}
			####1
			}
		}

	\expandafter\DeclareDocumentCommand\csname f#1\endcsname {+m}
		{%
		\IfValueTF{#4}
			{%Extra character
			\textcolor{#2}
			{\text{$\,^{(\incdisplaycounter{#4})}$}}
			\marginpar{\tiny\textcolor{#2}{
				{\text{\tiny $(\displaycounter{#4})$}}
				\text{\IfValueT{#5}{#5}
				####1}}}
			}
			{%No extra character
			\textcolor{#2}
			{$\,^{(\incdisplaycounter)}$}
			\marginpar{\tiny\textcolor{#2}{
				{\tiny $(\displaycounter)$}
				\text{\IfValueT{#5}{#5}
				####1}}}
			}
		}

	\expandafter\DeclareDocumentEnvironment{e#1} {}
		{%BEGIN
		\textcolor{#2}
		\bgroup
		\IfValueT{#5}
			{%
			#5
			}
		}
		{%END
		\egroup
		}
	}

\csname Show#1\endcsname

}

\definecolor{NAColor}{rgb}{1.0,0.0,0.0}
\definecolor{ProblemColor}{rgb}{0.7,0.1,0.7}
\definecolor{TBDColor}{rgb}{0.0,0.0,0.8}
\definecolor{MathColor}{rgb}{0.0,0.4,0.1}
\definecolor{NateColor}{rgb}{0.0,0.5,1.0}
\definecolor{MostafaColor}{rgb}{1.0,0.0,1.0}
\definecolor{RefColor}{rgb}{1.0,0.0,1.0}
\definecolor{LaterColor}{rgb}{1.0,0.0,1.0}

\DeclareComment{NA}{NAColor}{\EM{\star}}
\DeclareComment{TBD}{TBDColor}{\EM{!}}
\DeclareComment{PROBLEM}{ProblemColor}{\EM{!!}}(\bf)
\DeclareComment{MATH}{MathColor}{\EM{\#}}
\DeclareComment{NATE}{NateColor}{\EM{\square}}(N:)
\DeclareComment{MOSTAFA}{MostafaColor}{\EM{\square}}(M:)
\DeclareComment{REF}{RefColor}{\EM{(r)}}(Ref:)
\DeclareComment{LATER}{LaterColor}{\EM{(L)}}(Later:)

\DeclareDocumentCommand{\DeclareCounter}{m}%
{%
 	\ifcsname c@#1\endcsname%
	\else%
		\newcounter{#1}%
	\fi%
}

\DeclareDocumentCommand{\MyQED}{}{\qed}

\DeclareDocumentEnvironment{proof}{d() D[]{\MyQED} d<>}
{%BEGIN
\noindent\IfNoValueTF{#1}
{\emph{Proof.\!\!}}
{\emph{Proof\ #1.\ }}
}
{%END
\IfValueTF{#3}
	{%	
	\ExplSyntaxOff
	\RightJustify{\EM{#2_{#3}}}
	\vspace{1ex}
	}
	{
	\RightJustify{\EM{#2}}
	\vspace{1ex}
	}
}

\DeclareCounter{ProofLabelcOUntEr}
\DeclareDocumentCommand{\ProofLabel}{}{%
\addtocounter{ProofLabelcOUntEr}{1}
\label{cUrrEntProoflAbEl\arabic{ProofLabelcOUntEr}}
}

\DeclareDocumentCommand{\ProofRef}{D<>{1}}
{%
\ref{cUrrEntProoflAbEl\arabic{ProofcOUntEr#1}}
}
\DeclareDocumentCommand{\ProofCref}{D<>{1}}
{%
\cref{cUrrEntProoflAbEl\arabic{ProofcOUntEr#1}}
}

\DeclareDocumentEnvironment{Proof}{d() D[]{\MyQED} D<>{1}}
{%BEGIN
\DeclareCounter{ProofcOUntEr#3}%
\setcounter{ProofcOUntEr#3}{\value{ProofLabelcOUntEr}}%
\begin{proof}(#1)[#2]<\ProofRef<#3>>%
}
{%END
\end{proof}
}

\ExplSyntaxOn
\def\TheoremDepth{section}

\DeclareDocumentCommand{\DeclareTheorem}{m o m o}{%
\IfNoValueTF{#4}
	{%
	\IfNoValueTF{#2}
		{%
		\newtheorem{#1vArIAblE}{#3}
		}
		{%
		\newtheorem{#1vArIAblE}[#2vArIAblE]{#3}
		}
	}
	{%
	\newtheorem{#1vArIAblE}{#3}[#4]%
	}
\newtheorem*{#1vArIAblE*}{#3}

\DeclareDocumentEnvironment{#1}{o o}
	{%BEGIN
	\IfValueT{##2}%
		{
		\begin{spacing}{##2}
		}
	\IfValueTF{##1}
		{
		\begin{#1vArIAblE}[##1]
		}
		{
		\begin{#1vArIAblE}
		}
	\ProofLabel
	}
	{%END
	\IfValueT{##2}%
		{
		\end{spacing}{##2}
		}
	\end{#1vArIAblE}
	}

\DeclareDocumentEnvironment{#1*}{o o}
	{%BEGIN
	\IfValueT{##2}%
		{
		\begin{spacing}{##2}
		}
	\IfValueTF{##1}
		{
		\begin{#1vArIAblE*}[##1]
		}
		{
		\begin{#1vArIAblE*}
		}
	}
	{%END
	\IfValueT{##2}%
		{
		\end{spacing}{##2}
		}
	\end{#1vArIAblE*}
	}
}
\ExplSyntaxOff

\theoremstyle{plain}
\DeclareTheorem{theorem}{Theorem}[\TheoremDepth]
\DeclareTheorem{proposition}[theorem]{Proposition}
\DeclareTheorem{lemma}[theorem]{Lemma}
\DeclareTheorem{corollary}[theorem]{Corollary}
\DeclareTheorem{claim}[theorem]{Claim}

\theoremstyle{definition}
\DeclareTheorem{definition}[theorem]{Definition}
\DeclareTheorem{axiom}[theorem]{Axiom}
\DeclareTheorem{convention}[theorem]{Convention}

\theoremstyle{remark}
\DeclareTheorem{remark}[theorem]{Remark}
\DeclareTheorem{example}[theorem]{Example}
\DeclareTheorem{question}[theorem]{Question}
\DeclareTheorem{conjecture}[theorem]{Conjecture}

\begin{document}

\title{Algebraic Sunflowers}

\begin{abstract}
We study sunflowers within the context of  finitely generated substructures of ultrahomogeneous structures. In particular, we look at  bounds on how large a set system is needed to guarantee the existence of sunflowers of a given size. We show that if we fix the size of the sunflower, the function which takes the size of the substructures in our set system and outputs the size of a set system needed to guarantee a sunflower of the desired size can grow arbitrarily slowly. 
\end{abstract}

\author{Nathanael Ackerman}
\address{Harvard University,
Cambridge, MA 02138, USA}
\email{nate@aleph0.net}

\author{Mostafa Mirabi}
\address{Department of Mathematics and Computer Science, The Taft School, Watertown, CT 06795, USA}
\email{mmirabi@wesleyan.edu}

%,\ 54C99,\ 03B60,\ 03C75}

\maketitle

%%%%  %%%%  %%%%  %%%% %%%%  %%%%  %%%%  %%%%
\section{Introduction}
%%%%  %%%%  %%%%  %%%% %%%%  %%%%  %%%%  %%%%

A sunflower, also known as a $\Delta$-system, is a collection of sets such that any two pairs of distinct sets in the collection have a common intersection. When delving into the study of sunflowers a common focus lies in determining the existence of large sunflowers inside a given collection of sets. The existence of large sunflowers has wide ranging applications in computer science including in the study circuit lower bounds,  matrix multiplication, pseudo-randomness and cryptography. For a survey of the connections to computer science see \cite{MR4334977}.

Additionally, the existence of large sunflowers also has applications in mathematical logic, including in the study of forcing large generic structures. See for example \cite{Golshani}, \cite{Cohen}, and \cite{Cohen-Generic-with-Functions_AGM} (where sunflowers are called $\Delta$-systems). 

One of the first major result showing the existence of large sunflowers in any collection of finite sets of the same size was the ``Sunflower Lemma'' of \Erdos\ and Rado. 

\begin{lemma}[Sunflower Lemma \cite{Erdos-Rado}]
\label{Sunflower lemma}
Let $\SF\:\w \times \w \to \w$ be the function where, for $n, k \in \w$, $\SF(n, k)$ is the minimal value such that any set $\cF$ of sets of size $k$ with $|\cF| \geq \SF(n, k)$ has a sunflower of size $n$. 

Then for all $n\in\w$ and $k \geq 3$, 
\[
\SF(n, k) \leq k!(n-1)^k.
\]
\end{lemma}

The Sunflower Lemma asserts that any sufficiently large collection of sets of a fixed size is guaranteed to contain a large sunflower. Moreover, it provides an upper bound on how large the collection of sets must be. 

While the Sunflower Lemma gives an upper bound on the size of the sets needed to have a sunflower, it was clear even when it was proved that it was not optimal. This gave rise to the ``Sunflower Conjecture''.
\begin{conjecture}[Sunflower Conjecture \cite{Erdos-Rado}]
For all $n \in \w$ there is a constant $c_n$ such that
\[
(\forall k\geq 3)\,\, \SF(n, k) \leq c_n^k.
\]
\end{conjecture}

Over the years some progress has been made on this conjecture, with an important step forward coming recently in \cite{MR4334977}. However, the full conjecture remains open. 

Recently a new approach to studying the sunflower conjecture emerged in \cite{MR4622593} in which the authors study the sunflower conjecture in the context of set systems which satisfy various structural properties. Specifically there are three structural properties which they consider. First, they consider the case where the set system has finite VC dimension. Second, they consider the case where the set system has finite Littlestone dimension. And third, they consider the case when the set system is required to arises naturally from some background structure, e.g. the plane. 

In \cite{MR4622593}, the authors are able to provide improved bounds on $\SF$ when the set system has finite VC dimension as well as finite Littlestone dimension. But, in the case where the set systems arise from the background structure of the plane they are able to prove the full sunflower conjecture. 

In this paper, we aim to study the functions $\SF$ under the assumption that the set system come from some background structure. Specifically, we will work in a situation where there is some background ultrahomogeneous algebraic structure $\cM$ where every finitely generated substructure of $\cM$ is finite. Let $\SF[\cM]$ be the analog of $\SF$ restricted to finitely generated substructures of $\cM$. Our main result will show that for each $n \in \w$ and each function $\alpha\:\w \to \w$ there is an ultrahomogeneous structure $\cM$  such that $(\forall r \geq 3)\,  \SF[\cM](n, r)\leq \alpha(r)$, i.e. for every function, no matter how slowly it grows, we can find an algebraic structure such that the bound on the size of the set needed for an $n$-size sunflower grows slower than that function. Further, we can choose the structure $\cM$ to be totally categorical.

%%%%  %%%%  %%%%  %%%%  %%%%  %%%%
\subsection{Notation}
%%%%  %%%%  %%%%  %%%%  %%%%  %%%%

Let $\SeqSize[n](k)$ be the number of non-repeating sequences of size $k$ on $n$-elements. Suppose $p$ is a unary function. Let $\ol{p}(x) = \{p^k(x)\}_{0 < k \in \w}$.  If $X$ is a set we let $\Powerset(X)$ denote the collection of subsets of $X$. 

If $\alpha\:\w \to \w$ be a non-decreasing function whose limit is infinity let $\alpha^\circ\:\w \to \w$ be such that $\alpha^\circ(n) = \min\{k \st \alpha(k) \geq n\}$. 

Suppose $\Lang$ is a first order language and $\cM$ is a $\Lang$-structure. We say $\cM$ is an \defn{algebraic structure} when $\Lang$ has only function symbols. We say $\cM$ is \defn{locally finite} if for every finite set $A$ there is a finite substructure of $\cM$ containing $A$.  We let $\SubStr[k](\cM)$ be the collection of finitely generated substructures of $\cM$ of size at most $k$. We let $\SubStr(\cM) = \bigcup_{k \in \w} \SubStr[k](\cM)$. 

We say that a $\Lang$-structure $\cM$ is \defn{ultrahomogeneous} if, whenever $A, B$ are finitely generated substructures of $\cM$, and $\sigma\:A \to B$ is an isomorphism, then there is an automorphism of $\cM$ extending $\sigma$. We say $\cM$ has a the \defn{strong amalgamation property} or (SAP) if whenever $i_0\:A \to B$ and $j_0\:A \to C$ are embeddings of finitely generated substructures of $\cM$ there are embeddings $i_1\:B \to \cM$ and $j_1\:C \to \cM$ such that $i_1 \circ i_0 = j_1 \circ j_0$ and $i_1``[B] \cap j_1``[C] = (i_1 \circ i_0)``[A]$. Note $\cM$ has (SAP) precisely when its age does. 

An infinite structure is called \defn{totally categorical} whenever any two models of its first order theory of the same size are isomorphic. Structures which are totally categorical are particularly simple from a model theoretic point of view. For example, note that the theory of sets in the empty language is an example of totally categorical structures.

%%%%  %%%%  %%%%  %%%%  %%%%  %%%%  %%%%  %%%%
\section{Algebraic Sunflowers}
%%%%  %%%%  %%%%  %%%%  %%%%  %%%%  %%%%  %%%%

We now formalize the functions we want to consider. 

\begin{definition}
Suppose $\cM$ is a locally finite $\Lang$-structure and $X \subseteq \SubStr(\cM)$. 
\begin{itemize}
\item We say $X$ is \defn{uniform} if for all $A, B \in X$ there is an isomorphism between $A$ and $B$. 

\item We say $X$ is \defn{strongly uniform} if for all $A, B \in X$ there is an isomorphism from $A$ to $B$ which is the identity on $A\cap B$. 

\item We say $X$ is a \defn{sunflower} if for all $A_0, A_1, B_0, B_1 \in X$, $A_0 \cap A_1 = B_0 \cap B_1$. 

\end{itemize}

\end{definition}

\begin{definition}
Suppose $\Lang$ is a language and $\cM$ is a locally finite $\Lang$-structure. For each $n, k \in \w$ let $\SF[\cM](n, k)$ be the least $\ell$ such that whenever $X \subseteq \SubStr[k](\cM)$ and $|X| \geq \ell$ then there is an $X_0 \subseteq X$ such that $|X_0| = n$ and $X_0$ is a sunflower. 
\end{definition}

We will omit mention of $\cM$ when $\Lang$ is the empty language and $\cM$ is infinite. In this case $\SF(n, k)$ is the size of a set, all of whose elements have size at most $k$, which guarantees a sunflower of size $n$. Note that if $X$ is a finite set of elements all of which have size at most $k$ there is a set $X^*$ and a bijection $i\:X \to X^*$ such that 
\begin{itemize}
\item for all $x \in X^*$, $|x| = k$, 

\item for all $x \in X$, $x \subseteq i(x)$, 

\item for all $x \in X$, $(i(x) \setminus x) \cap \bigcup X = \emptyset$,

\item for any $X_0 \subseteq X$, $X_0$ is a sunflower if and only if $i``[X_0]$ is a sunflower. 
\end{itemize}

In particular, when dealing with sunflowers for sets without structure, there is no difference with considering sets all of the same size, or sets of bounded size. This is equivalent to saying that we loose no generality only considering strongly uniform sunflowers.

Note if there are no sunflowers of size $n$ consisting of elements of size at most $k$ then $\SF[\cM](n, k) = \infty$. By the Sunflower Lemma this can only happen if there are less than $\SF(n, k)$ substructures of $\cM$ of size at most $k$. Therefore it will make sense to restrict attention to structures where $\SubStr[k](\cM)$ is infinite for any sufficiently large $k$. In particular we will restrict attention to structures which are ultrahomogeneous with (SAP). 

Note that for any $\cM$ and any $n, k \in \w$ we have $\SF[\cM](n, k) \geq n$. The following straightforward lemma shows that for any fixed $k$ this lower bound can be achieved. 

\begin{proposition}
\label{Lower bounds for fixed k}
Let $\Lang = \{f\}$ where $f$ is a unary function. Then for every $k \in \w$ there is an $\Lang$-structure $\cM_k$ where 
\begin{itemize}
\item $\cM_k$ is ultrahomogeneous, 

\item $\cM_k$ has (SAP), 

\item $\cM_k$ is totally categorical, 

\item $(\forall n \in \w)\, \SF[\cM_{k}](n, k) = n$.
\end{itemize}
\end{proposition}
\begin{proof}
For $k \in \w$ let $C_k$ be the $\Lang$-structure with $k$-elements and where $f$ is a bijection. Let $\cM_k$ be the union of $\w$-many disjoint copies of $C_k$. 

Now let $\cF$ be any collection of sets substructures of $\cM_k$ of $k$. Each element of $\cF$ must be isomorphic to $C_k$ and hence no two elements of $\cF$ can have non-empty intersection. Therefore $\cF$ is a sunflower of size $|\cF|$. 
\end{proof}

\cref{Lower bounds for fixed k} showed us that if we fix the size of the subsets we are considering we can obtain the minimal bound on the size of sets needed to guarantee a large sunflower.  We did this by, intuitively, enlarging each point to be a substructure of the desired size. The problem of finding a sunflower among sets of size $k$ in $\cM_k$ then reduced to the problem of finding a sunflower of size $1$ among ordinary sets.  

Next look at what happens if instead of fixing the size of the sets we want to consider we fix the size of the sunflower we want to find. We then have to consider substructures of arbitrary sizes and so we need a more complicated construction. However, the underlying idea is similar. Our goal is to construct a structure whose substructures act like ordinary sets, but where the size of the substructure corresponding to as set of a given size grows arbitrarily fast. This will then let us use the Sunflower Lemma to ensure that for any fixed $n$ we can find structures $\cM$ where $\SF[\cM](n, k)$ grows arbitrarily slowly as a function of $k$. 

\begin{theorem}  
\label{Main Theorem}
There is a finite language $\Lang$, consisting of only unary and binary functions, such that whenever $\alpha\: \w \to \w$ is non-decreasing with $\alpha(0) \geq 3$ and $\lim_{n \to \infty} \alpha(n) = \infty$ there is a locally finite countable $\Lang$-structure $\cM_{\alpha}$ such that 
\begin{itemize}
\item[(a)] $\cM_{\alpha}$ is ultrahomogeneous, has (SAP), and is totally categorical, 

\item[(b)] if $\cA_0, \cA_1 \in \SubStr(\cM_{\alpha})$  with $|\cA_0| = |\cA_1|$, then $\cA_0 \cong \cA_1$, 

\item[(c)] if $\cF \subseteq \SubStr(\cM_{ \alpha})$ is finite, then there is a $\cF^* \subseteq \SubStr(\cM_{\alpha})$ and a bijection $i\:\cF \to \cF^*$ such that 
\begin{itemize}
\item for all $\cA_0, \cA_1 \in \cF^*$, $\cA_0 \cong \cA_1$, 

\item for all $\cA \in \cF$, $\cA \subseteq i(\cA)$, 

\item if $\cF_0 \subseteq \cF$, then $\cF_0$ is a sunflower if and only if $i[\cF_0]$ is a sunflower.
\end{itemize}
%These conditions guarantee that every set system can be extended to a uniform set system in being a sunflower is preserved.  

\item[(d)] all uniform sunflowers are strongly uniform, 

\item[(e)] $(\forall n\in \w)(\forall k)\, \SF[\cM_{\alpha}](n, k) \leq \alpha(k) \cdot (n-1)^{\alpha(k)}$.  
\end{itemize}
\end{theorem}
\begin{proof}
Let $\Lang_0 = \{c, s\}$ where $c, s$ are unary functions. Let $\beta\: \w \to \w$ be an increasing function. Let $\cN^-_\beta$ be the unique countable infinite $\Lang_0$-structure such that  
\begin{itemize}
\item $s$ is a bijection, 

\item for each $x \in \cN^-_\beta$ there is a $k\in \w$ such that $s^k(x) = x$. Let $i(x)$ be the least such $k$, 

\item for all $m \in \w$ there are infinitely many $x$ with $i(x) = \beta(m)$, 

\item for each $x$ there is an $m$ such that $i(x) = \beta(m)$. Let $m = j(x)$. 

\item $c^2 = c$, 

\item for all $x \in \cN^-_\beta$, there is a $k$ such that $s^k(x) = c(x)$. 

\end{itemize}

In other words, $\cN_\beta$ consists of infinitely many $s$-cycles of each length in $\{\beta(m)\st m \in \w\}$. Further each $s$ cycle has a unique distinguished element which is a fixed point of $c$.

Let $\Lang = \Lang_0 \cup \{p_0, p_1, a\}$ where $p_0, p_1$ are unary and $a$ is binary. 
Let $\cN_\beta$ be the $\Lang$-structure where 
\begin{itemize}
\item $\cN_\beta \rest[\Lang_0] = \cN_\beta^-$, 

\item if $c(x) \neq x$ then $p_0(x) = p_1(x) = x$ and for all $y \in \cN_\beta$, $a(x,y) = x$ and $a(y, x) = y$.  

\item if $j(x) = 1$ then $p_0(x) = p_1(x) = x$, 

\item if $c(x) = x$ and $j(x) = n+1$ then 
\begin{itemize}
\item $j(p_0(x)) = 1$, 

\item $j(p_1(x)) = n$, 

\item $c(p_0(x)) = p_0(x)$ and $c(p_1(x)) = p_1(x)$, 

\item $\ol{p_0}(x)$ has size $n+1$. 

\end{itemize}

\item if $c(x) = x$, $c(y) = y$, $j(x) = k$ and $j(y) = n$ then 
\begin{itemize}
\item if $k \neq 1$ then $a(x, y) = x$, 

\item if $k = 1$ and $x \in \ol{p_0}(y)$ then $a(x, y) = x$, 

\item if $k = 1$ and $x \not\in \ol{p_0}(y)$ then $j(a(x,y)) = n+1$,  $p_0(a(x,y)) = x$ and $p_1(a(x,y)) = y$, 

\item if $\ol{p}(x) = \ol{p}(y)$ then $x = y$, 

\end{itemize}

\end{itemize} 

Intuitively the structure $\cN_\beta$ is constructed as follows. First we divided $\cN_\beta$ into infinitely many equivalence classes, where all elements in the same equivalence class are some power of $s$ of each other. Furthermore, the size of the equivalence classes are precisely the range of $\beta$ (and there are infinitely many equivalence classes of each size). 

Inside each equivalence class we have a distinguished element which is a fixed point of $c$. These are the only elements for which the functions $p_0, p_1, a$ may be non-trivial. 

Now on these distinguished elements we have three functions, $p_0, p_1$ and $a$. If $x$ is an element in the equivalence class with $\beta(n)$ many elements then we want to think of $x$ as a tuple of length $n+1$, without repetitions, of equivalence classes with $\beta(0)$ many elements. We then think of $p_0$ as the projection onto the first element of the tuple and $p_1$ as the tuple that results from removing the first element of the tuple $x$. We then think of $a(x,y)$ as a pairing function which can add a single element $x$ to the front of the tuple $y$, but only if $x$ is not already in the tuple $y$. 

For any substructure $\cA \subseteq \cN_\beta$ let $b(\cA) = \{x \in \cA \st c(x) = x\text{ and }p_0(x) = x\}$, i.e. the collection of distinguished elements of equivalence classes with $\beta(0)$ many elements. The following are then immediate for (possibly infinite) substructures $\cA_0, \cA_1$ of $\cN_\beta$. 
\begin{itemize}
\item if $b(\cA_0) = b(\cA_1)$, then $\cA_0 = \cA_1$, 

\item if $i\:b(\cA_0) \to b(\cA_1)|$ is an bijection, then $i$ extends uniquely to an isomorphism $i^*\:\cA_0 \cong \cA_1$,

\item $b(\cA_0) \cap b(\cA_1) = b(\cA_0 \cap \cA_1)$.

\end{itemize}

Further, for any subset $A \subseteq b(\cN_\beta)$ there is a substructure $\cA \subseteq \cN_\beta$ with $b(\cA) = A$. Put together these imply $\cN_\beta$ satisfies (a) - (d).  All that is left is to choose a $\beta$ such that $\cN_\beta$ satisfies (e) with respect to $n$ and $\alpha$. 

Let $\gamma_\beta\:\w \to \w$ be the function where 
\[
\gamma_\beta(n) = \sum_{k \leq n} \SeqSize[n](k) \cdot \beta(k).
\]
Note $\gamma_\beta$ is increasing with $\lim_{n\to \infty} \gamma_\beta(n) = \infty$. We then have for any substructure $\cA$ that 
\[
|\cA| = \gamma_\beta(|s(\cA)|). 
\]

We also have from \cite{Erdos-Rado} that
\[
\SF[\cN_\beta](n, \gamma_\beta(k)) \leq k!(n-1)^k \leq k!(n-1)^{k!}
\]
and so
\[
\SF[\cN_\beta](n, k) \leq (\gamma_\beta^\circ(k))!(n-1)^{\gamma_\beta^\circ(k)!}
\]

But as $\beta$ goes to infinity and $\gamma_\beta^\circ$ is constant on the interval from $[\beta(k), \beta(k+1))$ we can find a $\beta_\alpha$ such that $(\forall k \in \w)\, \gamma_{\beta_\alpha}^\circ(k)! \leq \alpha(k)$. So if we let $\cM_\alpha = \cN_{\beta_\alpha}$, we have for all $n, k \in \w$
\[
\SF[\cM_{\alpha}](n, k) \leq 
\alpha(k) \cdot (n-1)^{\alpha(k)}
\]
as desired.  
\end{proof}

\begin{remark}
Note that when considering the function $\SF[\cM](n, k)$ we are considering finite substructures of $\cM$ whose size is bounded by $k$. One might instead wish to consider finite substructures which have a generating set whose size is bounded by $k$. However, the structure constructed in the proof of \cref{Main Theorem} shows that this approach would be problematic as every finitely generated substructure is generated by a single element. Specifically if $(A_i)_{i \in \w}$ is an increasing collection of finite substructures of $\cM_\alpha$ then $\{A_i\}_{i \in \w}$ is an infinite set system, each of which is generated by a single element, but which contains no sunflower of size $3$. 
\end{remark}

The following corollary, which shows that for any fixed $n$, the map $k \mapsto \SF[\cM](n, k)$ can be chosen to grow arbitrarily slowly, is then immediate from \cref{Main Theorem}. 

\begin{corollary}  
There is a finite language $\Lang$ such that whenever
\begin{itemize}
\item $n \in \w$, 

\item $\alpha\: \w \to \w$ is non-decreasing with $\alpha(0) \geq 3 \cdot (n-1)^3$ and $\lim_{n \to \infty} \alpha(n) = \infty$,
\end{itemize}
there is a locally finite countable $\Lang$-structure $\cM_{n, \alpha}$ such that 
\begin{itemize}
%\item $\cM_{\alpha}$ has finite Littlestone dimension, \fTBD{Show. What is the dimension}

\item $\cM_{n, \alpha}$ is ultrahomogeneous, has (SAP) and is totally categorical,

\item $(\forall k)\, \SF[\cM_{n, \alpha}](n, k) \leq \alpha(k)$.  
\end{itemize}
\end{corollary}

We have considered the growth rate of the function $\SF[\cM](n, k)$ where when we fix $n$ or $k$. We end with two conjectures about this function when neither coordinate is fixed. 

\begin{conjecture}
Suppose $f\:\w\to \w$ is a non-decreasing function which go to infinity. Then there is a constant $c_f \in \w$ such that whenever $\alpha\:\w \to \w$ is a non-decreasing function which goes to infinity and with $\alpha(0) \geq c_f$ there is a finite functional language $\Lang$ and a locally finite $\Lang$-structure $\cM_f$ such that
\[
(\forall k \in \w)\, \SF[\cM_f](f(k), k) \leq \alpha(k).
\]
\end{conjecture}

\begin{conjecture}
Suppose $g\:\w\to \w$ is a non-decreasing function which goes to infinity. There is a constant $d_g \in \w$ such that whenever $\alpha\:\w \to \w$ is a non-decreasing function which goes to infinity and with $\alpha(0) \geq d_g$ there is a finite functional language $\Lang$ and a locally finite $\Lang$-structure $\cN_g$ such that
\[
(\forall n \in \w)\, \SF[\cN_g](n, g(n)) \leq \alpha(n).
\]
\end{conjecture}

\bibliographystyle{amsnomr}
\bibliography{bibliography}

\end{document}